\documentclass[reqno]{amsart}

\usepackage{amsmath,amsthm,amssymb,amsfonts,mathtools}
\numberwithin{equation}{section} 
\theoremstyle{plain}
\newtheorem{theo+}           {Theorem}      [section]
\newtheorem{prop+}  [theo+]  {Proposition}
\newtheorem{coro+}  [theo+]  {Corollary}
\newtheorem{lemm+}  [theo+]  {Lemma}
\newtheorem{defi+}  [theo+]  {Definition}
\newtheorem{conj+}  [theo+]  {Conjecture}

\theoremstyle{definition}
\newtheorem{rema+}  [theo+]  {Remark}
\newtheorem{prob+}  [theo+]  {Problem}
\newtheorem{exam+}  [theo+]  {Example}

\newenvironment{theorem}{\begin{theo+}}{\end{theo+}}
\newenvironment{proposition}{\begin{prop+}}{\end{prop+}}

\newenvironment{lemma}{\begin{lemm+}}{\end{lemm+}}

\def\om{\omega}

\begin{document}

\baselineskip 18pt
\larger[2]
\title
[Determinant evaluations related to plane partitions] 
{New proofs of determinant evaluations\\ related to plane partitions}
\author{Hjalmar Rosengren}
\address
{Department of Mathematical Sciences
\\ Chalmers University of Technology\\SE-412~96 G\"oteborg, Sweden}
\email{hjalmar@chalmers.se}
\urladdr{http://www.math.chalmers.se/{\textasciitilde}hjalmar}

\thanks{Research  supported by the Swedish Science Research
Council (Vetenskapsr\aa det)}

\begin{abstract}
We give a new proof of a determinant evaluation due to Andrews, which has been used to enumerate cyclically symmetric and descending plane partitions. We also prove some related results, including a $q$-analogue of Andrews's determinant.
\end{abstract}

\maketitle

\section{Introduction}  

In 1979, George Andrews \cite{a} managed to evaluate the determinant
\begin{equation}\label{ad}\det_{0\leq m,n\leq N-1}\left(\delta_{mn}+\binom{x+m+n}{n}\right). \end{equation}
This allowed him to enumerate so called 
cyclically symmetric plane partitions (using the case $x=0$) and descending plane partitions ($x=2$). Andrews's proof, which takes up most of his 33 pages paper, amounts to partially working out the LU-factorization of the underlying matrix. This  requires both clever guess-work and creative use of hypergeometric series identities. Later, 
 Andrews and Stanton \cite{as} found a shorter proof, using what Krattenthaler \cite{kr} has called  ``a magnificient factorization theorem'' due to 
 Mills, Robbins and Rumsey \cite{mrr}. 
This proof can be simplified further, see \cite{ck,kr0,pw}.
It is our purpose to present a new and simple method for evaluating \eqref{ad}, using orthogonal polynomials. 
Roughly speaking, we compute \eqref{ad} by viewing each matrix element as the scalar product of two Meixner--Pollaczek polynomials, with respect to the orthogonality measure for certain Wilson polynomials, see \S \ref{cs}.

Our method can be used to prove further results. 
Ciucu, Eisenk\"olbl, Krattenthaler and Zare \cite{cekz} found that
\begin{equation}\label{cd}\det_{0\leq m,n\leq N-1}\left(\delta_{mn}+t \binom{x+m+n}{n}\right)\end{equation}
can be evaluated in closed form when $t^6=1$. Up to conjugation, this gives four cases, $t=\pm 1$ and $t=\pm e^{2\textup i\pi/3}$, the case $t=1$ being \eqref{ad}.
Our proof of \eqref{ad} can be modified to include the remaining three cases, see \S \ref{cvs}.

We will also obtain some new variations of \eqref{ad}.
Note that evaluating \eqref{cd} is equivalent to evaluating
\begin{equation}\label{ada}\det_{0\leq m,n\leq N-1}\left({m!}{(b)_m}\delta_{mn}+
t(b)_m(b)_n\,{}_2F_1\left(\begin{matrix}-m,-n\\b \end{matrix};1\right)\right).
\end{equation}
Indeed, by the Chu--Vandermonde summation, the ${}_2F_1$ equals $(b)_{m+n}/(b)_m(b)_n$.
Dividing the $n$th column by $n!(b)_n$ then gives \eqref{cd}, with  $x=b-1$.
Using continuous dual Hahn polynomials rather than Wilson polynomials, 
we will evaluate
\begin{equation}\label{bda}\det_{0\leq m,n\leq N-1}\left({m!}{(b)_m}\delta_{mn}+t2^{(m+n)/2}
(b)_m(b)_n\,{}_2F_1\left(\begin{matrix}-m,-n\\b \end{matrix};\frac 12\right)\right) \end{equation}
whenever $t^4=1$ (giving three non-equivalent cases, $t=\pm1$ and $t=\textup i$)
and 
\begin{equation}\label{td}\det_{0\leq m,n\leq N-1}\left({m!}{(b)_m}\delta_{mn}+t
3^{(m+n)/2}(b)_m(b)_n\,{}_2F_1\left(\begin{matrix}-m,-n\\b \end{matrix};\frac 13\right)\right) \end{equation}
whenever $t^3=-1$ (giving two non-equivalent cases,  $t=-1,\,e^{\textup i\pi/3}$), see \S \ref{nvs}.
These results have some relation to weighted enumeration of alternating sign matrixes. Indeed, as we explain further below, the case $b=1$, $t=e^{2\textup i\pi/3}$ of \eqref{ada} is related to the famous problem of enumerating alternating sign matrices of fixed size. Similarly, it follows from the work of Colomo and Pronko \cite{cp1} that the case 
$b=1$, $t=\textup i$ of \eqref{bda} is related to the $2$-enumeration of alternating sign matrices and the case $b=1$, $t=e^{ i\pi/3}$ of 
\eqref{td} to the $3$-enumeration.

Another problem, already discussed in \eqref{ad}, is to obtain a $q$-analogue of Andrews's determinant. In the combinatorially most interesting special cases, $x=0$ and $x=2$, such $q$-analogues were proved by Mills, Robbins and Rumsey \cite{mrr2}, thereby settling conjectures of Macdonald \cite{m} and Andrews \cite{a}.
However, until now nobody has found a $q$-analogue for the case of general $x$. We propose such an identity in Theorem \ref{qt} where, roughly speaking,  the summable ${}_2F_1$ in \eqref{ada} is replaced by a non-summable ${}_4\phi_3$. However,  our Theorem \ref{qt} does not contain the $q$-analogues found by Mills, Robbins and Rumsey. 
It would be interesting to prove those results using the method of the present work. Some other identities that one should look at are the  conjectured determinant evaluations given in \cite[Conj.\ 35--37]{kr2}.
For instance, the first of these, due to Guoce Xin, amounts to evaluating
$$\det_{0\leq m,n\leq N-1}\left(\delta_{mn}-\binom{x+m+n}{n+1}\right).$$

We would like to acknowledge that our main idea is contained
 in the work of Colomo and Pronko \cite{cp1,cp} on the six-vertex model.
In \cite{cp1}, these authors found a new determinant formula 
for
the partition function of the homogeneous six-vertex model with domain wall boundary conditions.  
At the ``ice point'', the Colomo--Pronko formula expresses
 the number of states 
of the model (on an $N\times N$ lattice) as
 \begin{equation}\label{zd}\det_{0\leq m,n\leq N-1}\left(-e^{\textup 2i\pi/3}\delta_{mn}+e^{\textup i\pi/3}\binom{m+n}{n}\right), \end{equation}
which is essentially  the  case $x=0,\, t=e^{2\textup i\pi/3}$ of \eqref{cd}.
 On the other hand, by the alternating sign matrix theorem  
\cite{ku,z}, the number of states is
\begin{equation}\label{an}\frac {1!4!7!\dotsm(3N-2)!}{N!(N+1)!(N+2)!\dotsm(2N-1)!}. \end{equation}
If we want to prove directly that \eqref{zd} equals \eqref{an} we can proceed as follows.
Let
$$\langle f,g\rangle_{\pm}=\operatorname{PV}\int_{-\infty}^\infty f(x)g(x)\,\frac{e^{\pm \pi x/6}}{\sinh(\pi x/2)}\,dx$$
and
$$\langle f,g\rangle=\langle f,g\rangle_{+}+\langle f,g\rangle_{-}=2\int_{-\infty}^\infty f(x)g(x)\,\frac{\sinh(\pi x/6)}{\sinh(\pi x/2)}\,dx. $$
Consider the determinant
$$D=\det_{0\leq m,n\leq N-1}\left(\langle p_m,p_n\rangle\right),$$
where $p_n$ is a monic polynomial of degree $n$.
By linearity in rows and columns, $D$ does not depend on the choice of $p_n$. 
Choosing $p_n$ as 
orthogonal with respect to the pairing $\langle\cdot,\cdot\rangle_+$, $D$ essentially reduces to \eqref{zd} \cite{cp1}. On the other hand,
choosing $p_n$ as orthogonal  with respect to  $\langle\cdot,\cdot\rangle$,  $D$ becomes diagonal 
and can thus be evaluated \cite{cp}. (Choosing $p_n(x)=x^n$ gives a Hankel determinant, which is the limit case of the Izergin--Korepin formula \cite{ick} used by Kuperberg \cite{ku} in his proof of \eqref{an}.)
 All our results are obtained by variations of this idea.

\section{Preliminaries on orthogonal polynomials}

For the benefit of the reader, we collect
 some fundamental facts on Wilson, continuous dual Hahn, Meixner--Pollaczek and Askey--Wilson polynomials, see  \cite{ks}. 
We refer to \cite{aar} or \cite{gr} for the standard  notation
for hypergeometric and basic hypergeometric series
 used throughout the paper.

The \emph{Wilson polynomials} are defined by
\begin{multline}\label{wp}W_n\left(x^2;a_1,a_2,a_3,a_4\right)=(a_1+a_2)_n(a_1+a_3)_n(a_1+a_4)_n\\
\times\,{}_4F_3\left(\begin{matrix}-n,a_1+a_2+a_3+a_4+n-1,a_1+\textup ix,a_1-\textup ix\\a_1+a_2,a_1+a_3,a_1+a_4\end{matrix};1\right). 
\end{multline}
This is a polynomial of degree $n$ in $x^2$ with leading coefficient
$$(-1)^n(a_1+a_2+a_3+a_4+n-1)_n. $$
If the parameters $a_k$ are all positive, Wilson polynomials
satisfy the orthogonality relation
\begin{multline}\label{wo}
\frac{\Gamma(a_1+a_2+a_3+a_4)}{2\pi\prod_{1\leq j<k\leq 4}\Gamma(a_j+a_k)}
  \int_0^\infty\left|\frac{\Gamma(a_1+\textup ix)\Gamma(a_2+\textup ix)\Gamma(a_3+\textup ix)\Gamma(a_4+\textup ix)}{\Gamma(2ix)}\right|^2\\
\times W_m\left(x^2;a_1,a_2,a_3,a_4\right)W_n\left(x^2;a_1,a_2,a_3,a_4\right)\,dx=h_n\,\delta_{mn},
 \end{multline}
where
$$h_n=h_n^\text{W}(a_1,a_2,a_3,a_4)=\frac{a_1+a_2+a_3+a_4-1}{a_1+a_2+a_3+a_4-1+2n}\frac{n!\prod_{1\leq j<k\leq 4}(a_j+a_k)_n}{(a_1+a_2+a_3+a_4-1)_n}. $$
Later, we will choose $a_1=0$. Then, the pole of the factor $\Gamma(a_1+ix)$ at $x=0$ is cancelled by the pole of $\Gamma(2ix)$. Thus, \eqref{wo} remains valid for $a_1=0$ as long as the other parameters are positive.

The \emph{continuous dual Hahn polynomials} are defined by
$$S_n\left(x^2;a_1,a_2,a_3\right)=(a_1+a_2)_n(a_1+a_3)_n\,{}_3F_2\left(\begin{matrix}-n,a_1+\textup ix,a_1-\textup ix\\a_1+a_2,a_1+a_3\end{matrix};1\right). $$
This is a polynomial of degree $n$ in $x^2$ with leading coefficient
$(-1)^n$.
If all $a_k$ are  positive, then
\begin{multline}\label{cho}
\frac{1}{2\pi\prod_{1\leq j<k\leq 3}\Gamma(a_j+a_k)}
  \int_0^\infty\left|\frac{\Gamma(a_1+\textup ix)\Gamma(a_2+\textup ix)\Gamma(a_3+\textup ix)}{\Gamma(2ix)}\right|^2\\
\times S_m\left(x^2;a_1,a_2,a_3\right)S_n\left(x^2;a_1,a_2,a_3\right)\,dx=h_n\,\delta_{mn},
 \end{multline}
where
$$h_n=h_n^\text{CDH}(a_1,a_2,a_3)=n!(a_1+a_2)_n(a_1+a_3)_n(a_2+a_3)_n. $$
Similarly as for \eqref{wo}, \eqref{cho} holds also for $a_1=0$ as long as the other parameters are positive.

The \emph{Meixner--Pollaczek polynomials} are defined by
\begin{equation}\label{mp}P_n^{(\lambda)}(x;\phi)=\frac{(2\lambda)_n}{n!}\,e^{\textup in\phi}\,{}_2F_1\left(\begin{matrix}-n,\lambda+\textup ix\\2\lambda\end{matrix};1-e^{-2\textup i\phi}\right).\end{equation}
This is a polynomial in $x$ of degree $n$ with leading coefficient
$$\frac{(2\sin\phi)^n}{n!}. $$
For $\lambda>0$ and $0<\phi<\pi$,
\begin{equation}\label{mpo}\frac{(2\sin\phi)^{2\lambda}}{2\pi\Gamma(2\lambda)}\int_{-\infty}^\infty
e^{(2\phi-\pi)x}\left|\Gamma(\lambda+\textup ix)\right|^2P_m^{(\lambda)}(x;\phi)P_n^{(\lambda)}(x;\phi)\,dx=h_n\delta_{mn}, \end{equation}
where
$$h_n=h_n^{\text{MP}}(\lambda)=\frac{(2\lambda)_n}{n!}. $$
We will need the expansion formula
\begin{equation}\label{me}
P_n^{(\lambda)}\left(x;\frac\pi 2+\phi\right)
=(-1)^n\frac{(2\lambda)_n}{n!}\sum_{k=0}^n\frac{(-n)_k}{(2\lambda)_k}\,(2\sin\phi)^{n-k}P_k^{(\lambda)}\left(x;\frac\pi 2-\phi\right),
\end{equation}
which can be proved by inserting \eqref{mp}, changing the order 
of summation and using the binomial theorem.

Finally, the \emph{Askey--Wilson polynomials}  are defined by
\begin{align*}p_n(\cos\theta;a_1,a_2,a_3,a_4|q)&=\frac{(a_1a_2,a_1a_3,a_1a_4;q)_n}{a_1^n}\\
&\quad\times\,{}_4\phi_3\left(\begin{matrix}q^{-n},a_1a_2a_3a_4q^{n-1},a_1e^{\textup i\theta},a_1e^{-\textup i\theta}\\a_1a_2,a_1a_3,a_1a_4\end{matrix};q,q\right).
\end{align*}
This is a polynomial in $\cos\theta$ of degree $n$ with leading coefficient
$$2^n(a_1a_2a_3a_4q^{n-1};q)_n. $$ 
It will be convenient to write the orthogonality 
 using $e^{\textup i\theta}$ rather than $\cos\theta$ as  integration variable.
Assuming 
\begin{equation}\label{apc}|q|,|a_1|,|a_2|,|a_3|,|a_4|<1,\end{equation}
 we have 
\begin{multline}\label{awo}\frac{(q;q)_\infty\prod_{1\leq j<k\leq 4}(a_ja_k;q)_\infty}{2(a_1a_2a_3a_4;q)_\infty}
\oint\frac{(z^2,z^{-2};q)_\infty}{(a_1z,a_1z^{-1},a_2z,a_2z^{-1},a_3z,a_3z^{-1},a_4z,a_4z^{-1};q)_\infty}\\
\times p_m\left(\frac{z+z^{-1}}2;a_1,a_2,a_3,a_4|q\right)p_n\left(\frac{z+z^{-1}}2;a_1,a_2,a_3,a_4|q\right)\,\frac{dz}{2\pi\textup iz}=h_n\delta_{mn}, \end{multline}
where the integral is over the positively oriented unit circle and
$$h_n=h_n^{\text{AW}}(a_1,a_2,a_3,a_4;q)=\frac{1-a_1a_2a_3a_4q^{-1}}{1-a_1a_2a_3a_4q^{2n-1}}\frac{(q;q)_n\prod_{1\leq j<k\leq 4}(a_ja_k;q)_n}{(a_1a_2a_3a_4q^{-1};q)_n}. $$
We will need the fact that \eqref{awo} remains valid when $a_1=1$, 
as long as the other conditions in \eqref{apc} hold.
The reason  is that
 the double zero of the factor $(a_1z,a_1z^{-1};q)_\infty$  at $z=1$ is cancelled by the double zero of $(z^2,z^{-2};q)_\infty$. 

\section{Proof of Andrews's determinant}
\label{cs}

We first explain the main idea behind our proof in general terms.
Suppose  we are given three symmetric bilinear forms
$\langle\cdot,\cdot\rangle_k$, $k=-1,0,1$,  which are defined on polynomials and related by
\begin{equation}\label{sps}\langle f,g\rangle_0=\langle f,g\rangle_1+\langle f,g\rangle_{-1}. \end{equation}
In the generic situation, there exist monic  polynomials
 $p_n^{(k)}$ of degree $n$, with $\langle p_m^{(k)},p_n^{(k)}\rangle_k=h_n^{(k)}\delta_{mn}$. We assume that this is the case for $k=0$ and $k=1$.

Consider the determinant
\begin{equation}\label{gd}D=\det_{0\leq m,n\leq N-1}\left(\langle p_m,p_{n}\rangle_0\right),  \end{equation}
with $p_n$  a monic polynomial of degree $n$. By linearity in rows and columns, $D$ is
 independent of the choice of $p_n$. In particular, choosing
 $p_n=p_n^{(0)}$ we find that $D=h_0^{(0)}h_1^{(0)}\dotsm h_{N-1}^{(0)}$. Choosing $p_n=p_n^{(1)}$ then  gives the key identity
\begin{equation}\label{gfd}\det_{0\leq m,n\leq N-1}\left(h_{m}^{(1)}\delta_{mn}+\langle p_{m}^{(1)},p_{n}^{(1)}\rangle_{-1}\right)=\prod_{n=0}^{N-1}h_{n}^{(0)}. \end{equation}

In the cases that we will consider,  the bilinear forms will be defined
by 
$$\langle f,g\rangle_{k}=\int_{-\infty}^\infty f(x)g(x)w_{k}(x)\,dx,\qquad k=-1,0,1,$$
where $w_0=w_1+w_{-1}$. In particular, we will show that if we take
\begin{subequations}\label{w}
\begin{align}\label{wpm}w_{\pm 1}(x)&=\frac{3^{(b+2)/2}}{4\pi\Gamma(b)}\, e^{\pm\pi x}\left|\Gamma\left(\frac b2+3\textup ix\right)\right|^2, \\
\label{wn}w_0(x)&=\frac{3^{(b+2)/2}}{2\pi\Gamma(b)}\, \cosh(\pi x)\left|\Gamma\left(\frac b2+3\textup ix\right)\right|^2, \end{align}
\end{subequations}
where $b>0$, then 
  \eqref{gfd} becomes Andrews's determinant evaluation \eqref{ad}, with $x=b-1$.

Let us first compute the polynomials $p_n^{(0)}$. Since  $w_0$ is even, we can write
$p_{2n}^{(0)}(x)=q_n(x^2)$, $p_{2n+1}^{(0)}(x)=x\,r_n(x^2)$, where 
$q_n$ and $r_n$ are  monic orthogonal polynomials on the positive half-line with weight $w_0$ and $x^2w_0$, respectively.

Recall that the gamma function satisfies the duplication formula
$$(2\pi)^{1/2}\Gamma(2x)=2^{2x-1/2}\Gamma(x)\Gamma\left(x+\frac 12\right), $$
the triplication formula
$$2\pi\Gamma(3x)=3^{3x-1/2}\Gamma(x)\Gamma\left(x+\frac 13\right)\Gamma\left(x+\frac 23\right) $$
and the reflection formula, which we write as
$$ \Gamma\left(\frac 12+\textup i x\right)\Gamma\left(\frac 12-\textup i x\right)=\frac{\pi}{\cosh(\pi x)}.$$
Combining these identities, one readily writes
\begin{align*}w_0(x)&=\frac{3^{3b/2}}{32\pi^3\Gamma(b)}\left|\frac{\Gamma(\textup ix)\Gamma(\textup ix+b/6)\Gamma(\textup ix+b/6+1/3)\Gamma(\textup ix+b/6+2/3)}{\Gamma(2\textup ix)}\right|^2\\
&=\frac{\Gamma(a_1+a_2+a_3+a_4)}{4\pi\prod_{1\leq j<k\leq 4}\Gamma(a_j+a_k)}\left|\frac{\Gamma(a_1+\textup ix)\Gamma(a_2+\textup ix)\Gamma(a_3+\textup ix)\Gamma(a_4+\textup ix)}{\Gamma(2\textup ix)}\right|^2,
 \end{align*}
with
$(a_1,a_2,a_3,a_4)=(0,b/6,b/6+1/3,b/6+2/3)$. Comparing this with \eqref{wo}, we find that 
\begin{subequations}\label{pnn}
\begin{equation}p_{2n}^{(0)}(x)=\frac{(-1)^n}{(b/2+n)_n}\,W_n\left(x^2;0,\frac b6,\frac b6+\frac 13,\frac b6+\frac 23\right) \end{equation}
and that
\begin{align}\notag h_{2n}^{(0)}&=\frac 1{(b/2+n)_n^2}\,h_n^\text W\left(0,\frac b6,\frac b6+\frac 13,\frac b6+\frac 23\right)\\
\label{pen}&=\frac{n!(b/2)_n(b/2)_{3n}(b+1)_{3n}}{3^{6n}(b/2)_{2n}(b/2+1)_{2n}}. \end{align}

Since $x^2|\Gamma(ix)|^2=|\Gamma(ix+1)|^2$, we can also write
$$x^2w_0(x)=C\frac{\Gamma(b_1+b_2+b_3+b_4)}{4\pi\prod_{1\leq j<k\leq 4}\Gamma(b_j+b_k)}\left|\frac{\Gamma(b_1+\textup ix)\Gamma(b_2+\textup ix)\Gamma(b_3+\textup ix)\Gamma(b_4+\textup ix)}{\Gamma(2\textup ix)}\right|^2, $$
with
$(b_1,b_2,b_3,b_4)=(1,b/6,b/6+1/3,b/6+2/3)$ and
$$C=\frac{b_2b_3b_4}{b_2+b_3+b_4}=\frac{b(b+4)}{2^2\cdot 3^3}. $$
It follows that
\begin{equation}p_{2n+1}^{(0)}(x)=\frac{(-1)^n}{(b/2+n+1)_n}\,x\,W_n
\left(x^2;1,\frac b6,\frac b6+\frac 13,\frac b6+\frac 23\right)
 \end{equation}
and 
\begin{align}\notag h_{2n+1}^{(0)}&=\frac{b(b+4)}{2^2\cdot 3^3(b/2+n+1)_n^2}\,h_n^\text W\left(1,\frac b6,\frac b6+\frac 13,\frac b6+\frac 23\right)\\
\label{pon}&=\frac{b(b+4)}{4}\frac{n!(b/2+1)_n(b/2+3)_{3n}(b+1)_{3n}}{ 3^{6n+3}(b/2+1)_{2n}(b/2+2)_{2n}}. \end{align}
\end{subequations}

As for the polynomials $p_n^{(\pm 1)}$, it follows from \eqref{mpo} that 
$$p_n^{(\pm 1)}(x)=\frac{n!}{3^{3n/2}}\,P_n^{(b/2)}\left(3x,\frac{\pi}2\pm \frac\pi 6\right)$$
and that
\begin{equation}\label{pn}h_n^{(\pm 1)}=\langle p_m^{(\pm 1)},p_n^{(\pm 1)}\rangle_{\pm 1}=\frac{(n!)^2}{2\cdot 3^{3n}}\,h_n^{\text{MP}}(b/2)=\frac{n!\,(b)_n}{2\cdot 3^{3n}}. \end{equation}
To compute $\langle p_m^{(1)},p_n^{(1)}\rangle_{-1}$, we use 
  \eqref{me} to expand
$$p_n^{(1)}=(-1)^n(b)_n\sum_{k=0}^n\frac{(-n)_k}{k!(b)_k}\,3^{3(k-n)/2}p_k^{(-1)}. $$
It follows that  
\begin{equation}\label{pp}\langle p_m^{(1)},p_n^{(1)}\rangle_{-1}=\frac{(-1)^{m+n}(b)_m(b)_n}{2\cdot 3^{3(m+n)/2}}\sum_{k=0}^{\min(m,n)}\frac{(-m)_k(-n)_k}{k!(b)_k}
=\frac{(-1)^{m+n}(b)_{m+n}}{2\cdot 3^{3(m+n)/2}}, \end{equation}
where the final step is the Chu--Vandermonde summation.

By \eqref{pn} and \eqref{pp}, the general determinant identity \eqref{gfd}
is now reduced to
$$
\det_{0\leq m,n\leq N-1}\left(\frac{m!\,(b)_m}{2\cdot3^{3m}}\,\delta_{mn}+(-1)^{m+n}\frac{(b)_{m+n}}{2\cdot 3^{3(m+n)/2}}\right)
=\prod_{n=0}^{[(N-1)/2]}h_{2n}^{(0)}\prod_{n=0}^{[(N-2)/2]}h_{2n+1}^{(0)},
$$
with $h_n^{(0)}$  as in \eqref{pnn}.
Multiplying the $n$th row and $n$th column with
$(-1)^n2^{1/2} 3^{3n/2}$, for each $n$, we arrive at the following result.

\begin{theorem}[Andrews]\label{at} The following determinant evaluation holds: 
\begin{multline*}\det_{0\leq m,n\leq N-1}\big({m!} \,{(b)_m}\delta_{mn}+{(b)_{m+n}}\big)=2^N\left(\frac{b(b+4)}{4}\right)^{\left[\frac N2\right]}\\
\times\prod_{n=0}^{[(N-1)/2]}\frac{n!(b/2)_n(b/2)_{3n}(b+1)_{3n}}{(b/2)_{2n}(b/2+1)_{2n}}\prod_{n=0}^{[(N-2)/2]}\frac{n!(b/2+1)_n(b/2+3)_{3n}(b+1)_{3n}}{(b/2+1)_{2n}(b/2+2)_{2n}}.
\end{multline*}
\end{theorem}

Dividing the $n$th column by $n!\,(b)_n$ and writing
$$\frac{(b)_{m+n}}{n!\,(b)_n}=\binom{b+m+n-1}{n},$$
we see that Theorem \ref{at} is indeed 
equivalent to the evaluation of \eqref{ad}.

\section{The CEKZ variations}
\label{cvs}

We will now modify our proof to cover the three variations of Andrews's determinant discovered by Ciucu, Eisenk\"olbl, Krattenthaler and Zare \cite{cekz}.

For the first  variation, we take
\begin{subequations}\label{sw}
\begin{align}\label{swpm}w_{\pm 1}(x)&=\frac{3^{(b+2)/2}}{4\pi\Gamma(b)}\, e^{\pm\pi x}\Gamma\left(\frac b2+3\textup ix+1\right)\Gamma\left(\frac b2-3\textup ix\right), \\
\label{swn}w_0(x)&=\frac{3^{(b+2)/2}}{2\pi\Gamma(b)}\, \cosh(\pi x)
\Gamma\left(\frac b2+3\textup ix+1\right)\Gamma\left(\frac b2-3\textup ix\right)
. \end{align}
\end{subequations}
In other words, $w_k$ are obtained by multiplying the weights in \eqref{w}
with
$$\frac{\Gamma(b/2+3\textup ix+1)}{\Gamma(b/2+3\textup ix)}=3i\left(x-\frac{\textup i b}6\right).
 $$

Recall that, in general,  if $p_n$ are monic orthogonal polynomials with
$$\int p_m(x)p_n(x)\,d\mu(x)=h_n\delta_{mn},$$ 
then
$$\tilde p_n(x)=\frac{p_{n+1}(x)-\frac{p_{n+1}(a)}{p_n(a)}\,p_n(x)}{x-a} $$
are monic orthogonal polynomials with
$$\int \tilde p_m(x)\tilde p_n(x)\,(x-a)\,d\mu(x)=\tilde h_n\delta_{mn},$$ 
where
$$\tilde h_n=-\frac{p_{n+1}(a)}{p_n(a)}\,h_n.$$

In the case at hand, it follows that 
$$ h_n^{(0)}\bigg|_{w_0 \text{ as in \eqref{swn}}}=-3\textup i\frac{p_{n+1}^{(0)}(\textup ib/6)}{p_n^{(0)}(\textup ib/6)}\,h_n^{(0)}\bigg|_{w_0 \text{ as in \eqref{wn}}}, $$
where the quantities on the right-hand side are given in 
 \eqref{pnn}. Applying the explicit formula \eqref{wp}
with $a_1$ and $a_2$ interchanged, both ${}_4F_3$:s reduce to a single term, and we find that
\begin{align*}p_{2n}^{(0)}\left(\frac{\textup i b}6\right)\Bigg|_{w_0 \text{ as in \eqref{wn}}}&=\frac{(-1)^n(b/6)_n(b/3+1/3)_n(b/3+2/3)_n}{(b/2+n)_n},\\
 p_{2n+1}^{(0)}\left(\frac{\textup i b}6\right)\Bigg|_{w_0 \text{ as in \eqref{wn}}}&=\frac{\textup i b}{6}\frac{(-1)^n(b/6+1)_n(b/3+1/3)_n(b/3+2/3)_n}{(b/2+n+1)_n}. \end{align*}
After simplification, this gives
\begin{subequations}\label{hts}
\begin{align}h_{2n}^{(0)}\bigg|_{w_0 \text{ as in \eqref{swn}}}&=\frac{ b}{2}\frac{n!(b/2+1)_n(b/2+1)_{3n}(b+1)_{3n}}{3^{6n}(b/2+1)_{2n}^2}, \\
h_{2n+1}^{(0)}\bigg|_{w_0 \text{ as in \eqref{swn}}}&=\frac{b(b+1)(b+4)}{2}\frac{n!(b/2+1)_n(b/2+3)_{3n}(b+3)_{3n}}{3^{6n+4}(b/2+2)_{2n}^2}.\end{align}
\end{subequations}

The remaining quantities that we need  can be obtained from the following Lemma. We formulate it so as to cover also some cases needed in \S \ref{nvs}.

\begin{lemma}\label{mml}
For $b$ and $t$ positive and $-\pi/2<\phi<\pi/2$, define the pairing
$$\langle p,q\rangle_{\phi}=\frac{t(2\cos\phi)^{b+1}}{2\pi\Gamma(b+1)}
\int_{-\infty}^\infty p(x)q(x)\,e^{ 2\phi t x}\Gamma\left(\frac b2+t\textup ix+1\right)
\Gamma\left(\frac b2-t\textup ix\right)\,dx.
 $$
Then, the rescaled Meixner--Pollaczek polynomials
$$p_n(x)=\frac{n!}{(2t\cos\phi)^n}\,P_n^{((b+1)/2)}\left(tx-\frac{\textup i} 2;\frac\pi 2+\phi\right) $$
are monic and satisfy the orthogonality relation
\begin{equation}\label{rmo}\langle p_m,p_n\rangle_\phi=\frac{e^{\textup i\phi}n!(b+1)_n}{(2t\cos\phi)^{2n}}\,\delta_{mn} \end{equation}
as well as
\begin{equation}\label{rme}\langle p_m,p_n\rangle_{-\phi}=e^{-\textup i\phi}\left(-\frac{\tan\phi}t\right)^{m+n}(b+1)_m(b+1)_n\,{}_2F_1\left(\begin{matrix}-m,-n\\b+1 \end{matrix};\frac 1{4\sin^2\phi}\right).
\end{equation}
\end{lemma}

\begin{proof}
Consider integrals of the form
\begin{equation}\label{gi}\frac{t(2\cos\phi)^{b+1}}{2\pi\Gamma(b+1)}
\int_{-\infty}^\infty p(x)\,e^{ 2\phi t x}\Gamma\left(\frac b2+t\textup ix+1\right)
\Gamma\left(\frac b2-t\textup ix\right)\,dx\end{equation}
with $p$ a polynomial. If we replace $x\mapsto x+\textup i/2t$ and 
 then
shift the contour of integration back to the real line, the value of the integral does not change. This is true  since the contour does not cross any poles of 
the integrand
and since,
 by \cite[Cor.\ 1.4.4]{aar}, for large values of $|\operatorname{Re} x|$
one may estimate
$$\left|\Gamma\left(\frac b2+1+t\textup ix\right)
\Gamma\left(\frac b2-t\textup ix\right)\right|\leq
C|\operatorname{Re} x|^be^{-\pi t|\operatorname{Re}x|}$$
uniformly in any vertical strip.
Thus, making also a further change of variables $x\mapsto x/t$, we find that
 \eqref{gi} equals
$$ \frac{(2\cos\phi)^{b+1}e^{\textup i\phi}}{2\pi\Gamma(b+1)}
\int_{-\infty}^\infty p\left(\frac{x+\textup i/2}t\right)\,e^{ 2\phi  x}\left|\Gamma\left(\frac {b+1}2+\textup ix+1\right)\right|^2\,dx.$$
The orthogonality \eqref{rmo} then follows from \eqref{mpo}. 
Moreover,  \eqref{me} gives
$$p_n^{(\phi)}=(-1)^n(b+1)_n\sum_{k=0}^n\frac{(-n)_k}{k!(b+1)_k}\left(\frac{\tan\phi}t\right)^{n-k}p_k^{(-\phi)}, $$
where we indicate the $\phi$-dependence of the polynomials $p_n$. 
Combining this with
 \eqref{rmo}, with $\phi$ replaced by $-\phi$, gives \eqref{rme}. 
\end{proof}

In the case at hand, it follows from Lemma \ref{mml} that the monic orthogonal polynomials
 with respect to $ w_1$ are given by
$$ p_n^{(1)}(x)=\frac{n!}{3^{3n/2}}\,P_n^{((b+1)/2)}\left(3x-\frac{\textup i}2,\frac{2\pi}3\right)
$$
and that 
\begin{align}\label{hats} h_n^{(1)}&=\frac{e^{\textup i\pi/6}(b)_{n+1}n!}{2\cdot 3^{3n+1/2}},  \\
\label{sts}\langle p_m^{(1)},p_n^{(1)}\rangle_{-1}&=(-1)^{m+n}\frac{e^{-\textup i\pi/6}(b)_{m+n+1}}{2\cdot 3^{(3m+3n+1)/2}}.\end{align}
Plugging \eqref{hts}, \eqref{hats} and \eqref{sts} into \eqref{gfd}, replacing $b$ by $b-1$ and simplifying, we recover the following result.

\begin{theorem}[Ciucu, Eisenk\"olbl, Krattenthaler and Zare]
\label{ct}
One has
\begin{multline*}\det_{0\leq m,n\leq N-1}\big({m!} \,{(b)_m}\delta_{mn}-e^{2\textup i\pi/3}{(b)_{m+n}}\big)
=\left({e^{-\textup i\pi/6}}{\sqrt 3}\right)^N\left(\frac{b(b+3)}{3}\right)^{\left[\frac {N}2\right]}\\
\times
\prod_{n=0}^{[(N-1)/2]}\frac{n!\left(\frac{b+1}2\right)_n\left(\frac{b+1}2\right)_{3n}(b)_{3n}}{\left(\frac{b+1}2\right)_{2n}^2}\prod_{n=0}^{[(N-2)/2]}\frac{n!\left(\frac{b+1}2\right)_n\left(\frac{b+5}2\right)_{3n}(b+2)_{3n}}{\left(\frac{b+3}2\right)_{2n}^2}.
\end{multline*}
\end{theorem}

The second variation is obtained by choosing $w_1$ as in \eqref{swpm} but replacing
 $w_{-1}$ by its negative. We must then take
 \begin{equation}\label{tsp}w_0(x)=\frac{3^{(b+2)/2}}{2\pi\Gamma(b)}\, \sinh(\pi x)
\Gamma\left(\frac b2+3\textup ix+1\right)\Gamma\left(\frac b2-3\textup ix\right)
. \end{equation}
Since $(x+\textup ib/6)w_0(x)$ is odd, the 
  monic orthogonal polynomials with respect to 
 $w_0$ can be constructed as
$p_{2n}^{(0)}(x)=s_n(x^2)$, $p_{2n+1}^{(0)}(x)= (x+\textup ib/6)\,t_n(x^2)$, where 
$s_n$ are orthogonal on the positive half-line with respect to 
$(w_0(x)+w_0(-x))/2=xw_0(x)/(x-\textup i b/6)$
and $t_n$  orthogonal with respect to $x(x+\textup i b/6)w_0(x)$.

To identify these polynomials, we write
\begin{align*} \frac{x}{x-\textup ib/6}\, w_0(x)&=\frac{\textup i b}{2\cdot 3^{1/2}}\frac{\Gamma(a_1+a_2+a_3+a_4)}{4\pi\prod_{1\leq j<k\leq 4}\Gamma(a_j+a_k)}\\
&\quad\times\left|\frac{\Gamma(a_1+\textup ix)\Gamma(a_2+\textup ix)\Gamma(a_3+\textup ix)\Gamma(a_4+\textup ix)}{\Gamma(2\textup ix)}\right|^2,\\
x\left(x+\frac{\textup i b}6\right)w_0(x)&=\frac{\textup i b(b+1)(b+2)}{2\cdot 3^{7/2}}\frac{\Gamma(b_1+b_2+b_3+b_4)}{4\pi\prod_{1\leq j<k\leq 4}\Gamma(b_j+b_k)}\\
&\quad\times\left|\frac{\Gamma(b_1+\textup ix)\Gamma(b_2+\textup ix)\Gamma(b_3+\textup ix)\Gamma(b_4+\textup ix)}{\Gamma(2\textup ix)}\right|^2, \end{align*}
where
\begin{align*}(a_1,a_2,a_3,a_4)&=\left(\frac 12,\frac b6,\frac b6+\frac13,\frac b6+\frac 23\right),\\
(b_1,b_2,b_3,b_4)&=\left(\frac 12,\frac b6+\frac 13,\frac b6+\frac 23,\frac b6+1\right).\end{align*}
It follows that
\begin{align*}
p_{2n}^{(0)}(x)&=\frac{(-1)^n}{(b/2+n+1/2)_n}\,W_n\left(x^2;\frac 12,\frac b6,\frac b6+\frac13,\frac b6+\frac 23\right),\\
p_{2n+1}^{(0)}(x)&=\frac{(-1)^n}{(b/2+n+3/2)_n}\,\left(x+\frac{\textup ib}6\right)W_n\left(x^2;\frac 12,\frac b6+ \frac 13,\frac b6+\frac 23,\frac b6+1\right)
\end{align*}
and that
\begin{subequations}\label{hv}
\begin{align}
\notag h_{2n}^{(0)}&=\frac{\textup i b}{2\cdot 3^{1/2}(b/2+n+1/2)_n^2}\,h_n^{\text W}\left(\frac 12,\frac b6,\frac b6+\frac13,\frac b6+\frac 23\right)\\
&\label{hve}=\frac{\textup i b}{2}\frac{n!\left(\frac{b+1}2\right)_n\left(\frac{b+3}2\right)_{3n}(b+1)_{3n}}{3^{6n+1/2}\left(\frac{b+1}2\right)_{2n}\left(\frac{b+3}2\right)_{2n}},\\
\notag h_{2n+1}^{(0)}&= \frac 1{(b/2+n+3/2)_n^2}\frac{\textup i b(b+1)(b+2)}{2\cdot 3^{7/2}}\,h_n^{\text W}\left(\frac 12,\frac b6+\frac 13,\frac b6+\frac 23,\frac b6+1\right)\\
&=\frac{\textup i b(b+1)(b+2)}{2\cdot 3^{6n+7/2}}\frac{n!(b/2+3/2)_n(b/2+5/2)_{3n}(b+3)_{3n}}{(b/2+3/2)_{2n}(b/2+5/2)_{2n}}.
\end{align}
\end{subequations}

Since \eqref{hats} is still valid and \eqref{sts} holds up to a change of sign, 
 we conclude that
$$
\det_{0\leq m,n\leq N-1}\left(\frac{e^{\textup i\pi/6}m!\,(b)_{m+1}}{2\cdot3^{(3m+1)/2}}\,\delta_{mn}+(-1)^{m+n+1}\frac{e^{-\textup i\pi/6}(b)_{m+n+1}}{2\cdot 3^{(3m+3n+1)/2}}\right)=\prod_{n=0}^{N-1}  h_{n}^{(0)},$$
with  $h_n^{(0)}$ as in \eqref{hv}.
Replacing $b$ with  $b-1$ and simplifying, 
we arrive at the following  result.

\begin{theorem}[Ciucu, Eisenk\"olbl, Krattenthaler and Zare]
\label{ct2}
One has
\begin{multline*}\det_{0\leq m,n\leq N-1}\big({m!} \,{(b)_m}\delta_{mn}+e^{2\textup i\pi/3}{(b)_{m+n}}\big)
=e^{\textup i\pi N/3}\big(b(b+1)\big)^{\left[\frac {N}2\right]}\\
\times
\prod_{n=0}^{[(N-1)/2]}\frac{n!\left(\frac{b}2\right)_n\left(\frac{b+2}2\right)_{3n}(b)_{3n}}{\left(\frac{b}2\right)_{2n}\left(\frac{b+2}2\right)_{2n}}\prod_{n=0}^{[(N-2)/2]}\frac{n!\left(\frac{b+2}2\right)_n\left(\frac{b+4}2\right)_{3n}(b+2)_{3n}}{\left(\frac{b+2}2\right)_{2n}\left(\frac{b+4}2\right)_{2n}}.
\end{multline*}
\end{theorem}

For the final variation, we choose $w_1$  as in \eqref{wpm}, but replace the weight function $w_{-1}$ by its negative. We must then take
\begin{equation}\label{w2}w_0(x)=\frac{3^{(b+2)/2}}{2\pi\Gamma(b)}\sinh(\pi x)\left|\Gamma\left(\frac b2+3\textup ix\right)\right|^2.
\end{equation}
Since $\langle 1,1\rangle_0=0$, there does not exist a system  of orthogonal polynomials 
with respect to $w_0$. Thus, \eqref{gfd} is not  applicable.
However,  we can  compute the
determinant \eqref{gd} using orthogonal polynomials
with respect to the weight $xw_0(x)$.

Consider  more generally the determinant \eqref{gd}, when
the scalar product is given by integration against an odd weight function 
$w_0$. Suppose  there exist monic orthogonal polynomials $q_n$ with
$$\int_{-\infty}^\infty q_m(x^2)q_n(x^2)\,xw_0(x)\,dx=2\int_0^\infty q_m(x^2)q_n(x^2)\,xw_0(x)\,dx=c_n\delta_{mn}.  $$
Then, the monic polynomials
$p_{2n}(x)=q_n(x^2)$, $p_{2n+1}(x)=xq_n(x^2)$ 
satisfy
$$\langle p_m,p_n\rangle_0=\begin{cases}c_k, & \{m,n\}=\{2k,2k+1\},\\
0, & \text{else}.
\end{cases} $$
Choosing $p_n$ in this way, $D$ reduces to the block-diagonal determinant
$$\left|\begin{matrix}0&c_0&0&0&\dots&0\\
c_0&0&0&0&&\\
0&0&0&c_1&&\\
0&0&c_1&0&&\\
\vdots&&&&\ddots&\\
0&&&&&0
\end{matrix}\right|.$$
Thus, as a substitute for \eqref{gfd} we have
\begin{multline}\label{ofd}\det_{0\leq m,n\leq N-1}\left(h_{m}^{(1)}\delta_{mn}+\langle p_{m}^{(1)},p_{n}^{(1)}\rangle_{-1}\right)\\
=
\begin{cases}
(-1)^{N/2}\left(c_0c_1\dotsm c_{(N-2)/2}\right)^2, & N \text{ even},\\
0, & N \text{ odd}.
\end{cases}\end{multline}

In the case at hand,  we observe that
$$xw_0(x)\bigg|_{w_0 \text{ as in } \eqref{w2}}=\frac 1{3\textup i}\frac{x}{x-\textup ib/6} w_0(x)\bigg|_{w_0 \text{ as in } \eqref{tsp}}, $$
which gives
$$c_n=\frac 1{3\textup i}\,h_{2n}^{(0)}\bigg|_{\text{as in \eqref{hve}}}
=\frac{b}{2}\cdot\frac{n!\left(\frac{b+1}2\right)_n\left(\frac{b+3}2\right)_{3n}(b+1)_{3n}}{3^{6n+3/2}\left(\frac{b+1}2\right)_{2n}\left(\frac{b+3}2\right)_{2n}}. $$
Since \eqref{pn} holds and \eqref{pp} holds up to a change of sign,
\eqref{ofd} can be simplified to the following form.

\begin{theorem}[Ciucu, Eisenk\"olbl, Krattenthaler and Zare]
\label{ct3}
When $N$ is even, 
\begin{multline*}\det_{0\leq m,n\leq N-1}\big({m!} \,{(b)_m}\delta_{mn}-{(b)_{m+n}}\big)\\
=(-1)^{\frac N2}b^N\prod_{n=0}^{(N-2)/2}\left(\frac{n!\left(\frac{b+1}2\right)_n\left(\frac{b+3}2\right)_{3n}(b+1)_{3n}}{\left(\frac{b+1}2\right)_{2n}\left(\frac{b+3}2\right)_{2n}}\right)^2,
\end{multline*}
whereas if $N$ is odd the determinant  vanishes.
\end{theorem}

\section{Further variations}
\label{nvs}

It is natural to look for further interesting specializations of \eqref{gfd}. 
 We have not found any more cases that are as nice as Andrews's determinant in the sense that the quantities  $h_n^{(0)}$, $h_n^{(1)}$ and  $\langle p_m^{(1)},p_n^{(1)}\rangle_{-1}$ all factor completely. However, from the viewpoint of orthogonal polynomials, there are five particularly natural cases based on continuous Hahn polynomials rather than Wilson polynomials. As we mentioned in the introduction, some of these evaluations are related to weighted enumeration of alternating sign matrices.  
Since the computations are completely parallel to those in \S \ref{cs}, we will be rather brief.

In the first of these five cases, we choose the weight functions as
\begin{align}\label{fcw}w_{\pm 1}(x)&=\frac{2^{b/2}}{2\pi\Gamma(b)}\, e^{\pm\pi x}\left|\Gamma\left(\frac b2+2\textup ix\right)\right|^2, \\
\notag w_0(x)&=\frac{2^{b/2}}{\pi\Gamma(b)}\, \cosh(\pi x)\left|\Gamma\left(\frac b2+2\textup ix\right)\right|^2, \end{align}
with  $b>0$.

With $(a_1,a_2,a_3)=(0,b/4,b/4+1/2)$ and 
$(b_1,b_2,b_3)=(1,b/4,b/4+1/2)$,
\begin{align*}w_0(x)&=\frac 1{4\pi\Gamma(a_1+a_2)\Gamma(a_1+a_3)\Gamma(a_2+a_3)}\left|\frac{\Gamma(a_1+ix)\Gamma(a_2+ix)\Gamma(a_3+ix)}{\Gamma(2ix)}\right|^2, \\
x^2w_0(x)&=\frac{b(b+2)}{16}\cdot\frac 1{4\pi\Gamma(b_1+b_2)\Gamma(b_1+b_3)\Gamma(b_2+b_3)}\\
&\quad\times\left|\frac{\Gamma(b_1+ix)\Gamma(b_2+ix)\Gamma(b_3+ix)}{\Gamma(2ix)}\right|^2. \end{align*}
Exactly as in the proof of Theorem \ref{at}, it follows that 
\begin{align}\notag h_{2n}^ {(0)}(x)&=h_n^{\text{CDH}}\left(0,\frac b4,\frac b4+\frac 12\right)=\frac{n!\left(\frac{b+1}2\right)_n\left(\frac b2\right)_{2n}}{4^n},\\
\notag h_{2n+1}^ {(0)}(x)&=\frac{b(b+2)}{16}\,h_n^{\text{CDH}}\left(1,\frac b4,\frac b4+\frac 12\right)=b(b+2)\frac{n!\left(\frac{b+1}2\right)_n\left(\frac {b+4}2\right)_{2n}}{4^{n+2}},\\
\notag p_n^{(\pm 1)}(x)&=\frac{n!}{2^{3n/2}}\,P_n^{(b/2)}\left(2x,\frac\pi2\pm \frac{\pi}4\right),\\
\label{hh} h_n^{(1)}&=\frac{n!(b)_n}{2^{3n+1}},\\
\label{sh} \langle p_m^{(1)},p_n^{(1)}\rangle_{-1}&=\frac{(-1)^{m+n}(b)_m(b)_n}{2^{m+n+1}}
\,{}_2F_1\left(\begin{matrix}-m,-n\\b \end{matrix};\frac 12\right).
\end{align}
After simplification,  \eqref{gfd} then reduces to the following new identity.

\begin{theorem}\label{hdt}
The following determinant evaluation holds:
\begin{multline*}
\det_{0\leq m,n\leq N-1}\left({m!}{(b)_m}\delta_{mn}+
2^{(m+n)/2}(b)_m(b)_n\,{}_2F_1\left(\begin{matrix}-m,-n\\b \end{matrix};\frac 12\right)\right)=2^{N^2}\\
\times\left(\frac{b(b+2)}{8}\right)^{[N/2]}\prod_{n=0}^{[(N-1)/2]}n!\left(\frac{b+1}2\right)_n\left(\frac{b}2\right)_{2n}\prod_{n=0}^{[(N-2)/2]}n!\left(\frac{b+1}2\right)_n\left(\frac{b+4}2\right)_{2n}.
\end{multline*}
\end{theorem}

Next, we take 
\begin{align*}
w_{\pm 1}(x)&=\pm\frac{2^{b/2}}{2\pi\Gamma(b)}\, e^{\pm \pi x}\Gamma\left(\frac b2+2\textup ix+1\right)\Gamma\left(\frac b2-2\textup ix\right), \\
 w_0(x)&=\frac{2^{b/2}}{\pi\Gamma(b)}\, \sinh(\pi x)
\Gamma\left(\frac b2+2\textup ix+1\right)\Gamma\left(\frac b2-2\textup ix\right). \end{align*}
Similarly as in the proof of Theorem \ref{ct2}, we write
\begin{align*}\frac{x}{x-\textup ib/4}\,w_0(x)
&=\frac{\textup i b}{2}\frac{1}{4\pi\prod_{1\leq j<k\leq 3}\Gamma(a_j+a_k)}\left|\frac{\Gamma(a_1+\textup ix)\Gamma(a_2+\textup ix)\Gamma(a_3+\textup ix)}{\Gamma(2\textup ix)}\right|^2,\\
x\left(x+\frac{\textup ib}4\right)w_0(x)
&=\frac{\textup i b(b+1)(b+2)}{16}\frac{1}{4\pi\prod_{1\leq j<k\leq 3}\Gamma(b_j+b_k)}\\
&\quad\times\left|\frac{\Gamma(b_1+\textup ix)\Gamma(b_2+\textup ix)\Gamma(b_3+\textup ix)}{\Gamma(2\textup ix)}\right|^2,
 \end{align*}
where
$(a_1,a_2,a_3)=(1/2,b/4,b/4+1/2)$ and $(b_1,b_2,b_3)=(1/2,b/4+1/2,b/4+1)$,
and conclude that 
\begin{align}
 h_{2n}^{(0)}&=\frac{\textup i b}{2}\,h_n^\text{CDH}\left(\frac 12,\frac b4,\frac b4+\frac 12\right)
\label{hj}=\frac{\textup i b}{2^{2n+1}}\,n!\left(\frac{b+1}2\right)_n\left(\frac{b+2}2\right)_{2n},\\
\notag h_{2n+1}^{(0)}&=\frac{\textup i b(b+1)(b+2)}{16}\,h_n^\text{CDH}\left(\frac 12,\frac b4+\frac 12,\frac b4+1\right)\\
\notag &=\frac{\textup i b(b+1)(b+2)}{2^{2n+4}}\,n!\left(\frac{b+3}2\right)_n\left(\frac{b+4}2\right)_{2n}.
\end{align}
Moreover, it follows from Lemma \ref{mml}  that
\begin{align*}h_n^{(1)}&=\frac{e^{\textup i\pi/4}n!(b)_{n+1}}{2^{3n+3/2}},\\
\langle p_m^{(1)},p_n^{(1)}\rangle_{-1}&=(-1)^{m+n+1}e^{-\textup i\pi/4}\frac{b(b+1)_m(b+1)_n}{2^{m+n+3/2}}\,{}_2F_1\left(\begin{matrix}-m,-n\\b+1 \end{matrix};\frac 12\right). \end{align*}
In the resulting instance of \eqref{gfd},
we replace $b$ by $b-1$ and simplify to obtain the following result.

\begin{theorem}\label{ht}
The following determinant evaluation holds:
\begin{multline*}
\det_{0\leq m,n\leq N-1}\left({m!}{(b)_m}\delta_{mn}+\textup i
2^{(m+n)/2}(b)_m(b)_n\,{}_2F_1\left(\begin{matrix}-m,-n\\b \end{matrix};\frac 12\right)\right)=e^{\frac{\textup i\pi N}4}2^{\frac{N(2N-1)}2}\\
\times\left(\frac{b(b+1)}{4}\right)^{[N/2]}\prod_{n=0}^{[(N-1)/2]}n!\left(\frac{b}2\right)_n\left(\frac{b+1}2\right)_{2n}\prod_{n=0}^{[(N-2)/2]}n!\left(\frac{b+2}2\right)_n\left(\frac{b+3}2\right)_{2n}.
\end{multline*}
\end{theorem}

Next, we choose $w_1$ as in \eqref{fcw} but replace $w_{-1}$ by its negative. Then,
$$w_0(x)=\frac{2^{b/2}}{\pi\Gamma(b)}\, \sinh(\pi x)\left|\Gamma\left(\frac b2+2\textup ix\right)\right|^2.$$
Exactly  as in the proof of Theorem \ref{ct3}, we  find that
 \eqref{ofd} holds with
$$c_n=\frac 1{2\textup i}\,h_{2n}^{(0)}\bigg|_{\text{ as in }\eqref{hj}}=
\frac{b}{2^{2n+2}}\,n!\left(\frac{b+1}2\right)_n\left(\frac{b+2}2\right)_{2n},
 $$
$h_n^{(1)}$ as in \eqref{hh} and $\langle p_m^{(1)},p_n^{(1)}\rangle_{-1}$ as in \eqref{sh} apart from a change of sign. After simplification, we obtain the following determinant evaluation. 

\begin{theorem}\label{xt}
When $N$ is even,
\begin{multline*}
\det_{0\leq m,n\leq N-1}\left({m!}{(b)_m}\delta_{mn}-
2^{(m+n)/2}(b)_m(b)_n\,{}_2F_1\left(\begin{matrix}-m,-n\\b \end{matrix};\frac 12\right)\right)\\
=(-1)^{\frac N2}2^{\frac{N(2N-3)}2}b^N
\prod_{n=0}^{(N-2)/2}\left(n!\left(\frac{b+1}2\right)_n\left(\frac{b+2}2\right)_{2n}
\right)^2,\end{multline*}
whereas if  $N$ is odd, the determinant vanishes.
\end{theorem}

We now turn to determinant evaluations of the form \eqref{td}. Let
$$w_{\pm 1}(x)=\pm\frac{3}{4\pi\Gamma(b)}\,e^{\pm 2\pi x}\Gamma\left(\frac b2+3\textup ix+1\right)\Gamma\left(\frac b2-3\textup ix\right),
 $$
$$w_{0}(x)=\frac{3}{2\pi\Gamma(b)}\,\sinh(2\pi x)\Gamma\left(\frac b2+3\textup ix+1\right)\Gamma\left(\frac b2-3\textup ix\right).
 $$
We then have
\begin{align*}\frac{x}{x-\textup ib/6}\,w_0(x)
&=\frac{\textup i 3^{1/2}b}{2}\frac{1}{4\pi\prod_{1\leq j<k\leq 3}\Gamma(a_j+a_k)}\\
&\quad\times\left|\frac{\Gamma(a_1+\textup ix)\Gamma(a_2+\textup ix)\Gamma(a_3+\textup ix)}{\Gamma(2\textup ix)}\right|^2,\\
x\left(x+\frac{\textup ib}6\right)w_0(x)
&=\frac{\textup i b(b+1)(b+2)}{2\cdot 3^{3/2}}\frac{1}{4\pi\prod_{1\leq j<k\leq 3}\Gamma(b_j+b_k)}\\
&\quad\times\left|\frac{\Gamma(b_1+\textup ix)\Gamma(b_2+\textup ix)\Gamma(b_3+\textup ix)}{\Gamma(2\textup ix)}\right|^2,
 \end{align*}
where $(a_1,a_2,a_3)=(b/6,b/6+1/3,b/6+2/3)$, $(b_1,b_2,b_3)=(b/6+1/3,b/6+2/3,b/6+1)$. As before, it follows that
\begin{align}
\label{hx} h_{2n}^{(0)}&=\frac{\textup i 3^{1/2}b}{2}\,h_n^\text{CDH}\left(\frac b6,\frac b6+\frac 13,\frac b6+\frac 23\right)=\frac{\textup i b}{2\cdot 3^{3n-1/2}}\,n!(b+1)_{3n},\\
\notag h_{2n+1}^{(0)}&=\frac{\textup i b(b+1)(b+2)}{2\cdot 3^{3/2}}\,h_n^\text{CDH}
\left(\frac b6+\frac 13,\frac b6+\frac 23,\frac b6+1\right)\\
\notag &=\frac{\textup i b(b+1)(b+2)}{2\cdot 3^{3n+3/2}}\,n!(b+3)_{3n}
\end{align}
and, using  Lemma \ref{mml},
\begin{align*}
h_n^{(1)}&=\frac{e^{\textup i\pi/3}n!(b)_{n+1}}{2\cdot 3^{2n}}, \\
\langle p_m^{(1)},p_n^{(1)}\rangle_{-1}&=(-1)^{m+n+1}e^{-\textup i\pi/3}\frac{b(b+1)_m(b+1)_n}{2\cdot 3^{(m+n)/2}}\,{}_2F_1\left(\begin{matrix}-m,-n\\b+1 \end{matrix};\frac 13\right).\end{align*}
After replacing $b$ by $b-1$,
 \eqref{gfd} can be simplified to the following form.

\begin{theorem}\label{yt}
The following determinant evaluation holds:
\begin{multline*}
\det_{0\leq m,n\leq N-1}\left({m!}{(b)_m}\delta_{mn}+e^{\textup i\pi/3}
3^{(m+n)/2}(b)_m(b)_n\,{}_2F_1\left(\begin{matrix}-m,-n\\b \end{matrix};\frac 13\right)\right)\\
=e^{\frac{\textup i\pi N}6}3^{\frac{N(N+1)}4}
\left(\frac{b(b+1)}{\sqrt 3}\right)^{[N/2]}\prod_{n=0}^{[(N-1)/2]}n!(b)_{3n}\prod_{n=0}^{[(N-2)/2]}n!(b+2)_{3n}.
\end{multline*}
\end{theorem}

Finally, we choose
$$w_{\pm 1}(x)=\pm\frac{3}{4\pi\Gamma(b)}\,e^{\pm 2\pi x}\left|\Gamma\left(\frac b2+3\textup ix\right)\right|^2,
 $$
$$w_{0}(x)=\frac{3}{2\pi\Gamma(b)}\,\sinh(2\pi x)\left|\Gamma\left(\frac b2+3\textup ix\right)\right|^2.
 $$
As before, we find that
 \eqref{ofd} holds with
 \begin{align*}c_n&=\frac 1{3\textup i}\,h_{2n}^{(0)}\bigg|_{\text{ as in }\eqref{hx}}=\frac{b}{2\cdot 3^{3n+1/2}}\,n!(b+1)_{3n},\\
 h_n^{(1)}&=\frac{1}{2\cdot 3^{2n}}\,n!(b)_n,\\
 \langle p_m^{(1)},p_n^{(1)}\rangle_{-1}&=\frac{(-1)^{m+n+1}(b)_m(b)_n}{2\cdot 3^{(m+n)/2}}\,{}_2F_1\left(\begin{matrix}-m,-n\\b \end{matrix};\frac 13\right),
 \end{align*}
which gives the following identity after simplification.

\begin{theorem}\label{tdc}
When $N$ is even,
\begin{multline*}
\det_{0\leq m,n\leq N-1}\left({m!}{(b)_m}\delta_{mn}-
3^{(m+n)/2}(b)_m(b)_n\,{}_2F_1\left(\begin{matrix}-m,-n\\b \end{matrix};\frac 13\right)\right)\\
=(-1)^{N/2}3^{N^2/4}b^N\prod_{n=0}^{(N-2)/2}\big(n!(b+1)_{3n}\big)^2
\end{multline*}
whereas if $N$ is odd the determinant vanishes.
\end{theorem}

It may be instructive to summarize the results obtained so far. 
We have considered 
 weight functions  
$$w_{\pm 1}(x)=\frac{l(2\cos(\pi k/2l))^b}{4\pi\Gamma(b)}\,(\pm 1)^\delta e^{\pm k\pi x}\Gamma\left(\frac b2+l\textup ix+\varepsilon\right)\Gamma\left(\frac b2-l\textup ix\right), $$
which are normalized so that $w_0=w_1+w_{-1}$ has total mass $1$ when $\delta=\varepsilon= 0$,
and the parameters are as in the following table:\\[1ex]

\begin{center}
\begin{tabular}{lrrrr}
&$k$&$l$&$\delta$&$\varepsilon$\\
Theorem \ref{at}&1&3&0&0\\
Theorem \ref{ct}&1&3&0&1\\
Theorem \ref{ct2}&1&3&1&1\\
Theorem \ref{ct3}&1&3&1&0\\
\hline
Theorem \ref{hdt}&1&2&0&0\\
Theorem \ref{ht}&1&2&1&1\\
Theorem \ref{xt}&1&2&1&0\\
\hline
Theorem \ref{yt}&2&3&1&1\\
Theorem \ref{tdc}&2&3&1&0
\end{tabular}\ .
\end{center}
\vspace*{1ex}

The case $(k,l,\delta,\varepsilon)=(1,2,0,1)$, which may appear to be missing, merely gives the complex conjugate of
Theorem \ref{ht}.

\section{A $q$-analogue of Andrews's determinant}

To find a $q$-analogue of Andrews's determinant, it is natural
to replace the Wilson polynomials in \eqref{pnn}  by
Askey--Wilson polynomials. More precisely, a natural starting point would be
 to combine the polynomials
$$p_n(x;1,b,bq,bq^2|q^3),\qquad p_n(x;q^3,b,bq,bq^2|q^3) $$
to a single orthogonal system. This is indeed possible, within
 the framework of orthogonal Laurent polynomials on the unit circle.

Throughout this section, we  write
 $$\om=e^{2\pi\textup i/3}.$$

\begin{lemma}\label{all}
For $|q|,|b|<1$, let
$$\langle f,g\rangle_0=\frac{(q^3;q^3)_\infty(b,b^2q;q)_\infty}{(b^3q^3;q^3)_\infty}
\oint f(z)g(z)\frac 1{1+z}\frac{(z^2,z^{-2};q^3)_\infty}{(z,z^{-1};q^3)_\infty
(bz,bz^{-1};q)_\infty}\frac{dz}{2\pi\textup iz},
 $$
with integration over the positively oriented unit circle. Then, the Laurent polynomials 
$$p_{2n}^{(0)}(z)=\frac 1{(b^3q^{3n};q^3)_n}\,p_n\left(\frac{z+z^{-1}}2;1,b,bq,bq^2|q^3\right), $$
$$p_{2n+1}^{(0)}(z)=\frac {z-1}{(b^3q^{3n+3};q^3)_n}\,p_n\left(\frac{z+z^{-1}}2;q^3,b,bq,bq^2|q^3\right) $$
satisfy the orthogonality relations
$$\langle p_m^{(0)},p_n^{(0)}\rangle_0=h_n^{(0)}\delta_{mn}, $$
where
\begin{subequations}\label{aln}
\begin{align}
\label{he}
h_{2n}^{(0)}&=\frac{(q^3,b^3;q^3)_n(b,b^2q;q)_{3n}}{(b^3,b^3q^3;q^3)_{2n}}, \\
\label{ho}h_{2n+1}^{(0)}&=-\frac{(1-b)(1-bq^2)}{(1-\om bq)(1-\om^2 bq)}\frac{(q^3,b^3q^3;q^3)_n(bq^3,b^2q;q)_{3n}}{(b^3q^3,b^3q^6;q^3)_{2n}}. \end{align}
\end{subequations}
\end{lemma}
 
\begin{proof}
In the integral defining $\langle p_m^{(0)},p_n^{(0)}\rangle_0$, write
$$\oint f(z)\,\frac{dz}{2\pi\textup iz}=\frac 12\oint \left(f(z)+f(z^{-1})\right)\frac{dz}{2\pi\textup iz}. $$
When  $m$ and $n$ are both even, the integrand is invariant under $z\mapsto z^{-1}$ apart from the factor
$$\frac 1{1+z}+\frac 1{1+z^{-1}}=1. $$
The integral then reduces to \eqref{awo}, and we obtain the desired orthogonality with
$$h_{2n}^{(0)}=\frac{h_n^{\text{AW}}(1,b,bq,bq^2|bq^3)}{(b^3q^{3n};q^3)_n^2}, $$
which agrees with \eqref{he}.
When  $m$ and $n$ have different parity, we get an integral containing the factor 
$$\frac{z-1}{1+z}+\frac{z^{-1}-1}{1+z^{-1}}=0, $$
so the orthogonality is obvious.
Finally, when $m$ and $n$ are both odd, we encounter the factor
$$\frac{(z-1)^2}{1+z}+\frac{(z^{-1}-1)^2}{1+z^{-1}}=-(1-z)(1-z^{-1}). $$
We now  observe that $(1-z)(1-z^{-1})$ times the normalized orthogonality measure for $p_n(x;1,b,bq,bq^2|q^3)$ is  the corresponding measure for $p_n(x;q^3,b,bq,bq^2|q^3)$, apart from the multiplier
$$\frac{(1-b_2)(1-b_3)(1-b_4)}{1-b_2b_3b_4}\Bigg|_{b_2=b,\,b_3=bq,\,b_4=bq^2}=\frac{(1-b)(1-bq^2)}{(1-\om bq)(1-\om^2bq)}, $$
 and conclude that
$$h_{2n+1}^{(0)}=-\frac{(1-b)(1-bq^2)}{(1-\om bq)(1-\om^2bq)}\frac{h_n^{\text{AW}}(q^3,b,bq,bq^2|bq^3)}{(b^3q^{3n+3};q^3)_n^2}, $$
which agrees with \eqref{ho}.
\end{proof}

Note that  $p_k^{(0)}$ is a linear combination of the first $k+1$ terms in the sequence 
$$1,\,z,\,z^{-1},\,z^2,\,z^{-2},\dots. $$
Moreover, the coefficient of the $(k+1)$st term is $1$. If we let these two properties define a \emph{monic Laurent polynomial of degree $k$}, then 
the discussion leading to 
\eqref{gfd} remains valid if 
 ``polynomial'' is replaced throughout by ``Laurent polynomial''. 

To apply this modified version of \eqref{gfd} we must split the orthogonality measure 
 for $p_n^{(0)}$ in two parts.
This will be achieved by the following version of Watson's quintuple product identity \cite{w}.
The fact that this fundamental result is applicable is a strong indication that we are doing something natural.

\begin{lemma}\label{wl}
The following identity holds:
\begin{align}\notag\frac 1{1+z}\frac{(z^2,z^{-2};q^3)_\infty}{(z,z^{-1};q^3)_\infty}
&=\frac{1-\omega}3\frac{(q;q)_\infty}{(q^3;q^3)_\infty}\\
\label{wq}&\quad\times\left(
{(qz\om,\om^2/z;q)_\infty-\om^2(qz\om^2,\om/z;q)_\infty}
\right).\end{align}
\end{lemma}

\begin{proof}
The left-hand side of \eqref{wq} can be expressed as
$$\frac 1 z\,(-z,-q/z;q^3)_\infty(q^3z^2,q^3/z^2;q^6)_\infty. $$
By the quintuple product identity, as given in \cite[Ex.\ 5.6]{gr}, 
the Laurent expansion of this function is 
\begin{equation}\label{la}\frac 1{(q^3;q^3)_\infty}\sum_{n=-\infty}^\infty(-1)^nq^{\binom{3n}2}z^{3n-1}(1+zq^{3n}). \end{equation}
On the other hand, by the triple product identity \cite[Eq.\ (1.6.1)]{gr},
the right-hand side of \eqref{wq} has Laurent expansion
\begin{equation}\label{lb}\frac{1-\om}{3(q^3;q^3)_\infty}\sum_{n=-\infty}^\infty(-1)^nq^{\binom{n+1}2}z^n(\om^n-\om^{2n+2}). \end{equation}
It is easily verified that \eqref{la} and \eqref{lb} agree.
\end{proof}

Let us introduce the notation
$$\mu_{a,b}(f)=\frac{(q,b^2;q)_\infty}{(qab,b/a;q)_\infty}\oint f(z)\frac{(qaz,1/az;q)_\infty}{(bz,b/z;q)_\infty}\frac{dz}{2\pi\textup i z}. $$
Then, using  Lemma \ref{wl},  the bilinear form introduced in Lemma \ref{all}
splits as in \eqref{sps}, with
\begin{align*}\langle f,g\rangle_1&=\frac{(1-\om)(1-b\om^2)}{3(1+b)}\,\mu_{\om,b}(fg), \\
\langle f,g\rangle_{-1}&=\frac{(1-\om^2)(1-b\om)}{3(1+b)}\,\mu_{\om^2,b}(fg).
\end{align*}
To proceed, we need the following result.

\begin{proposition}\label{olp}
Let
$$P_n^{(a,b;q)}(z)=
z^{-\left[\frac n2\right]}\,{}_2\phi_1\left(\begin{matrix}q^{-n},b/a\\q^{1-n}/ab\end{matrix};q,\frac{qz}{b}\right) $$
and let 
$$C_n=C_n^{(a,b;q)}=\begin{cases}1, & n \text{\emph{ even}},\\ 
a^n(b/a;q)_n/(ab;q)_n, & n \text{\emph{ odd}}.
\end{cases}$$
Then,
$P_n^{(a,b;q)}/C_n$ is a  monic Laurent polynomial of degree $n$.
 For $|b|,\,|q|<1$ we then have the orthogonality relation
\begin{equation}\label{lo}\mu_{a,b}\left(P_m^{(a,b;q)}P_n^{(a,b;q)}\right)=h_n\delta_{mn}, \end{equation}
where 
$$h_n=(-1)^n
a^{2\left[\frac n2\right]}\frac{(q,b^2,b/a;q)_n}{(ab;q)_n(abq;q)_{2[n/2]}(b/a;q)_{2[(n+1)/2]}}.
$$
Moreover, 
 \begin{align}\notag\mu_{c,b}\left(P_m^{(a,b;q)}P_n^{(a,b;q)}\right)&=
a^{\left[\frac{m}2\right]+\left[\frac{n}2\right]}
\frac{(b^2;q)_m(b^2;q)_n(qc/a;q)_{\left[\frac{m}2\right]+\left[\frac{n}2\right]}(a/c;q)_{\left[\frac{m+1}2\right]+\left[\frac{n+1}2\right]}}{(ab;q)_m(ab;q)_n(qbc;q)_{\left[\frac{m}2\right]+\left[\frac{n}2\right]}(b/c;q)_{\left[\frac{m+1}2\right]+\left[\frac{n+1}2\right]}}\\
&\label{sp}\quad\times
\,{}_4\phi_3\left(\begin{matrix}q^{-m},q^{-n},ab,b/a\\b^2,q^{-\left[\frac{m}2\right]-\left[\frac{n}2\right]}a/c,q^{1-\left[\frac{m+1}2\right]-\left[\frac{n+1}2\right]}c/a\end{matrix};q,q\right).
\end{align}
\end{proposition}

Before proving Proposition \ref{olp}, we point out that \eqref{lo} is equivalent to a result of Pastro \cite{p}.
Namely, if we let $p_n(z)=z^{[n/2]}P_n^{(a,b;q)}(z)$, then $p_n$ is a monic polynomial of degree $n$. Moreover, \eqref{lo} means that 
$$\mu_{a,b}(z^{-k}p_n(z))=0,\qquad k=0,1,\dots,n-1.$$ 
It follows that
$$\oint p_m(z)\,p_n(1/z)\frac{(qaz,1/az;q)_\infty}{(bz,b/z;q)_\infty}\frac{dz}{2\pi\textup i z}=0,\qquad m\neq n $$
or, if the parameters are such that $p_n$ have real coefficients,  
$$\oint p_m(z)\,\overline{p_n(z)}\,\frac{(qaz,1/az;q)_\infty}{(bz,b/z;q)_\infty}\frac{dz}{2\pi\textup i z}=0,\qquad m\neq n. $$
It is easily seen that this orthogonal system on the unit circle is equivalent to the one introduced by Pastro.

\begin{proof}[Proof of \emph{Proposition \ref{olp}}]
It is straight-forward to check that
 $P_n^{(a,b;q)}/C_n$ is a  monic Laurent polynomial of degree $n$.
To prove \eqref{sp}, we use \cite[Eq.\ (III.6--7)]{gr} to write
\begin{align*}P_n^{(a,b;q)}(z)&=
z^{-\left[\frac n2\right]}\left(\frac ab\right)^n\frac{(b^2;q)_n}{(ab;q)_n}\,{}_3\phi_2\left(\begin{matrix}q^{-n},b/a,bz\\b^2,0\end{matrix};q,q\right)\\ 
&=
\frac{z^{\left[\frac {n+1}2\right]}}{b^n}\frac{(b^2;q)_n}{(ab;q)_n}\,{}_3\phi_2\left(\begin{matrix}q^{-n},ab,b/z\\b^2,0\end{matrix};q,q\right).
\end{align*}
These expressions also clarify the relation to Meixner--Pollaczek polynomials
\eqref{mp}. Expressing  $P_m^{(a,b;q)}$ using the first
of these identities and  $P_n^{(a,b;q)}$ using the second one gives
\begin{multline*}\mu_{c,b}\left(P_m^{(a,b;q)}P_n^{(a,b;q)}\right)
=\frac{a^m}{b^{m+n}}\frac{(q,b^2;q)_\infty(b^2;q)_m(b^2;q)_n}{(qbc,b/c;q)_\infty(ab;q)_m(ab;q)_n}\\
\times\sum_{k=0}^{m}\sum_{l=0}^n\frac{(q^{-m},b/a;q)_k}{(q,b^2;q)_k}\frac{(q^{-n},ab;q)_l}{(q,b^2;q)_l}q^{k+l}
\oint z^{\left[\frac{n+1}2\right]-\left[\frac{m}2\right]}\frac{(qcz,1/cz;q)_\infty}{(bq^kz,bq^l/z;q)_\infty}\frac{dz}{2\pi\textup i z}. \end{multline*}
By Ramanujan's summation \cite[Eq.\ (II.29)]{gr}, 
$$\frac{(q,b^2,qcz,1/cz;q)_\infty}{(qbc,b/c,bq^kz,bq^l/z;q)_\infty}
=\frac{(b^2;q)_{k+l}}{(b/c;q)_k(qbc;q)_l}\sum_{x=-\infty}^\infty
\frac{(q^{1-k}c/b;q)_x}{(q^{l+1}bc;q)_x}\,(q^kbz)^x
$$ 
in the annulus $|q^{k+l}b|<|z|<|b^{-1}|$, which is consistent with our conditions
 on the parameters. Since
only the term with $x=[m/2]-[(n+1)/2]$ contributes to the integral, we 
conclude that
\begin{multline*}\mu_{c,b}\left(P_m^{(a,b;q)}P_n^{(a,b;q)}\right)
=\frac{a^m}{b^{\left[\frac{m+1}2\right]+\left[\frac{n+1}2\right]+n}}\frac{(b^2;q)_m(b^2;q)_n(qc/b;q)_{\left[\frac{m}2\right]-\left[\frac{n+1}2\right]}}{(ab;q)_m(ab;q)_n(qbc;q)_{\left[\frac{m}2\right]-\left[\frac{n+1}2\right]}}\\
\times\sum_{k=0}^{m}\sum_{l=0}^n\frac{(b^2;q)_{k+l}}{(b^2;q)_k(b^2;q)_l}\frac{(q^{-m},b/a;q)_k}{(q,q^{-\left[\frac{m}2\right]+\left[\frac{n+1}2\right]}b/c;q)_k}\frac{(q^{-n},ab;q)_l}{(q,q^{\left[\frac{m}2\right]-\left[\frac{n+1}2\right]+1}bc;q)_l}q^{k+l}. \end{multline*}
Applying  \cite[Eq.\ (II.7)]{gr}
in the form 
$$\frac{(b^2;q)_{k+l}}{(b^2;q)_k(b^2;q)_l}=\sum_{x=0}^{\min(k,l)}\frac{(q^{-k},q^{-l};q)_x}{(q,b^2;q)_k}\left(b^2q^{k+l}\right)^x, $$
 replacing $(k,l)$ by $(k+x,l+x)$ and  changing
the order of summation gives after simplification
 \begin{align*}&\mu_{c,b}\left(P_m^{(a,b;q)}P_n^{(a,b;q)}\right)
=\frac{a^m}{b^{\left[\frac{m+1}2\right]+\left[\frac{n+1}2\right]+n}}\frac{(b^2;q)_m(b^2;q)_n(qc/b;q)_{\left[\frac{m}2\right]-\left[\frac{n+1}2\right]}}{(ab;q)_m(ab;q)_n(qbc;q)_{\left[\frac{m}2\right]-\left[\frac{n+1}2\right]}}\\
&\qquad\times\sum_{x=0}^{\min(m,n)}\frac{(q^{-m},q^{-n},ab,b/a;q)_x}{(q,b^2,q^{-\left[\frac{m}2\right]+\left[\frac{n+1}2\right]}b/c,q^{1+\left[\frac{m}2\right]-\left[\frac{n+1}2\right]}bc;q)_x}\,q^{x(x+1)}b^{2x}\\
& \qquad\times\sum_{k=0}^{m-x}\frac{(q^{x-m},q^xb/a;q)_k}{(q,q^{x-\left[\frac{m}2\right]+\left[\frac{n+1}2\right]}b/c;q)_k}\,q^k\sum_{l=0}^{n-x}\frac{(q^{x-n},q^xab;q)_l}{(q,q^{1+x+\left[\frac{m}2\right]-\left[\frac{n+1}2\right]}bc;q)_l}\,q^l
. \end{align*}
Computing the inner sums using \cite[Eq.\ (II.6)]{gr} and simplifying, we 
arrive at \eqref{sp}.

To deduce \eqref{lo}, we observe that the right-hand side of \eqref{sp}
has the form
$$\sum_{x=0}^{\min(m,n)}(q^{x-[m/2]-[n/2]}a/c)_{m+n-2x}(\dotsm),$$ 
where the missing factors are regular at $c=a$. 
When $c\rightarrow a$, the summand vanishes for
$$x\leq\min\left(\left[\frac m2\right]+\left[\frac n2\right],\left[\frac{m+1}2\right]+\left[\frac{n+1}2\right]-1\right)=\left[\frac{m+n-1}2\right].$$
 Thus, the range of summation can be reduced to
$$\left[\frac{m+n+1}2\right]\leq x\leq \min(m,n), $$
which is empty for $m\neq n$ and consists of the point $x=m=n$ otherwise.
After simplification, this gives \eqref{lo}.
\end{proof}

We can now conclude that 
\begin{multline*}\det_{0\leq m,n\leq N-1}\left(\frac{(1-\om)(1-b\om^2)}{3(1+b)}\mu_{\om,b}
\left(P_m^{(\om,b;q)}P_m^{(\om,b;q)}\right)\delta_{mn}\right.\\
\left.+\frac{(1-\om^2)(1-b\om)}{3(1+b)}\mu_{\om^2,b}
\left(P_m^{(\om,b;q)}P_n^{(\om,b;q)}\right)
\right)=\prod_{n=0}^{N-1}C_n^2h_n^{(0)}, \end{multline*}
where $h_n^{(0)}$ is as in \eqref{aln} and the remaining quantities are 
given in Lemma \ref{all} and Proposition \ref{olp}.
To simplify, we pull out the factor
$(1-\om)(1-b\om^2)/3(1+b)$ and multiply the $n$th row and  column with
$\om^{2[n/2]}(b\om ;q)_n$, for each $n$. We also write
\begin{equation}\label{xn}X_m=(-1)^m\frac{(b\om,b\om^2;q)_m}{(q b\om;q)_{2[m/2]}(b\om^2;q)_{2[(m+1)/2]}}
=\begin{cases}\displaystyle\frac{1-b\om}{1-b\om q^m}, & m \text{ even},\\[3mm]
\displaystyle-\frac{1-b\om}{1-b\om^2 q^m}, & m \text{ odd}.
\end{cases}
 \end{equation}
After simplification, we find the following $q$-analogue of Andrews's determinant.

\begin{theorem}\label{qt}
Let $\om=e^{2\pi\textup i/3}$ and let $X_m$ be as in \eqref{xn}. Then, the following determinant evaluation holds:
\begin{align*}
&
\det_{0\leq m,n\leq N-1}\Bigg(X_m{(q,b^2;q)_m}\,\delta_{mn}
-\om^2\frac{(b^2;q)_m(b^2;q)_n(q\om;q)_{\left[\frac{m}2\right]+\left[\frac{n}2\right]}(\om^2;q)_{\left[\frac{m+1}2\right]+\left[\frac{n+1}2\right]}}{(b\om^2;q)_{\left[\frac{m}2\right]+\left[\frac{n}2\right]+1}(qb\om;q)_{\left[\frac{m+1}2\right]+\left[\frac{n+1}2\right]-1}}\\
&\qquad\quad\times
\,{}_4\phi_3\left(\begin{matrix}q^{-m},q^{-n},b\om,b\om^2\\b^2,q^{-\left[\frac{m}2\right]-\left[\frac{n}2\right]}\om^2,q^{1-\left[\frac{m+1}2\right]-\left[\frac{n+1}2\right]}\om\end{matrix};q,q\right)\Bigg)\\
&\qquad=\left(\frac{3(1+b)}{(1-\om)(1-b \om^2)}\right)^N\left(-\frac{(1-b)(1-bq^2)(1-b\om^2)^2\om^2}{(1-qb\om)(1-qb\om^2)}\right)^{[N/2]}\\
&\qquad\quad\times\prod_{n=0}^{[(N-1)/2]}
\frac{\om^n(b\om;q)_{2n}^2(q^3,b^3;q^3)_n(b,b^2q;q)_{3n}}{(b^3,b^3q^3;q^3)_{2n}}\\
&\qquad\quad\times\prod_{n=0}^{[(N-2)/2]}\frac{\om^{2n}(qb\om^2;q)_{2n}^2(q^3,b^3q^3;q^3)_n(bq^3,b^2q;q)_{3n}}{(b^3q^3,b^3q^6;q^3)_{2n}}.
\end{align*}
\end{theorem}

Replacing $b$ by $q^{b/2}$ and letting $q\rightarrow 1$, the ${}_4\phi_3$ reduces to a summable ${}_2F_1$ and we recover Andrews's determinant evaluation. 
Incidentally, replacing  $b$ by $-q^{b/2}$ and letting $q\rightarrow 1$
gives Theorem \ref{tdc}.

 \end{document}